\documentclass[a4paper,reqno,11pt]{amsart}
\usepackage{amsmath}
\usepackage{amssymb}
\usepackage{cases}
\allowdisplaybreaks[4]

\newtheorem{thm}{Theorem}[section]

\newtheorem{lem}[thm]{Lemma}

\newtheorem{cor}[thm]{Corollary}

\newtheorem{conj}[thm]{Conjecture}
\newtheorem{claim}{Claim}

\theoremstyle{definition}
\newtheorem{definition}[thm]{Definition}

\theoremstyle{remark}
\newtheorem{remark}[thm]{Remark}

\numberwithin{equation}{section}

\newcommand{\bQ}{\mathbb{Q}}
\newcommand{\bP}{\mathbb{P}}

\newcommand\OO{{\mathcal{O}}}
\newcommand{\rounddown}[1]{\lfloor{#1}\rfloor}

\newcommand\Vol{\text{\rm Vol}}

\newcommand\mult{{\rm{mult}}}
\newcommand\Nklt{{\rm{Nklt}}}
\newcommand\lct{{\rm{lct}}}
\newcommand\ulct{{\rm{ulct}}}
\begin{document}

\title{Boundedness of anti-canonical volumes of singular log Fano threefolds}
\date{November 25, 2014}
\author{Chen Jiang}
\address{Graduate School of Mathematical Sciences, the University of Tokyo,
3-8-1 Komaba, Meguro-ku, Tokyo 153-8914, Japan.}
\email{cjiang@ms.u-tokyo.ac.jp}
 \curraddr{Kavli IPMU (WPI), UTIAS, The University of Tokyo, Kashiwa, Chiba 277-8583, Japan.}
\email{chen.jiang@ipmu.jp}

\thanks{The author was supported by Grant-in-Aid for JSPS Fellows (KAKENHI No. 25-6549) and Program for Leading Graduate  Schools, MEXT, Japan}

\begin{abstract}
We prove the Weak Borisov--Alexeev--Borisov Conjecture in dimension three which states that the anti-canonical volume of an $\epsilon$-klt log Fano pair  of dimension three is bounded from above.
\end{abstract} 

\maketitle
\pagestyle{myheadings} \markboth{\hfill  C. Jiang
\hfill}{\hfill Boundedness of anti-canonical volumes of singular log Fano $3$-folds\hfill}

\tableofcontents

\section{Introduction}
Throughout the article, we work over the field of complex numbers $\mathbb{C}$. We adopt the standard notations and definitions in \cite{KMM} and \cite{KM}, and will freely use them.
\begin{definition}
A {\it pair} $(X, \Delta)$ consists of  a normal projective variety $X$ and an effective
$\mathbb{Q}$-divisor $\Delta$ on $X$ such that
$K_X+\Delta$ is $\mathbb{Q}$-Cartier.   $(X, \Delta)$ is called a
\emph{log Fano pair} (resp. {\it weak log Fano pair}) if $-(K_X+\Delta)$ is ample (resp. nef and big). If $\dim X=2$, we will use {\it del Pezzo} instead of Fano. 
\end{definition}

\begin{definition}
Let $(X, \Delta)$ be a pair. Let $f: Y\rightarrow X$ be a log
resolution of $(X, \Delta)$, write
$$
K_Y =f^*(K_X+\Delta)+\sum a_iF_i,
$$
where $F_i$ are distinct prime divisors. The coefficient $a_i$ is called the {\it discrepancy} of $F_i$ with respect to $(X, \Delta)$, and denoted by $a_{F_i}(X, \Delta)$. For a real number $\epsilon \in [0,1]$, the
pair $(X,\Delta)$ is called
\begin{itemize}
\item[(a)] \emph{$\epsilon$-kawamata log terminal} (\emph{$\epsilon$-klt},
for short) if $a_i> -1+\epsilon$ for all $i$;

\item[(b)] \emph{$\epsilon$-log canonical} (\emph{$\epsilon$-lc}, for
short) if $a_i\geq  -1+\epsilon$ for all $i$;

\item[(c)] \emph{terminal} if  $a_i> 0$ for all $f$-exceptional divisors $F_i$. 
\end{itemize}
Usually we write $X$ instead of $(X,0)$ in the case $\Delta=0$.
\end{definition}
Note that $0$-klt (resp. $0$-lc) is just klt (resp. lc) in the usual sense. 

\begin{definition}
A variety $X$ is \emph{of $\epsilon$-Fano type} if there exists an effective $\bQ$-divisor $\Delta$ such that $(X, \Delta)$ is an $\epsilon$-klt log Fano pair. 
\end{definition}

We are mainly interested in the boundedness of  varieties of $\epsilon$-Fano type. 

\begin{definition}
A collection of varieties $\{X_\lambda\}_{\lambda\in \Lambda}$ is
said to be \emph{bounded} if there exists $h:\mathcal{X}\rightarrow
S$ a projective morphism between schemes  of finite type such that for
each $X_\lambda$, $X_\lambda\simeq \mathcal{X}_s$ for some $s\in S$.
\end{definition}

Our motivation  is the following BAB
Conjecture due to A. Borisov, L. Borisov, and V. Alexeev.

\begin{conj}[BAB Conjecture]
Fix $0<\epsilon<1$ and an integer $n>0$. 
Then the set of all
$n$-dimensional  varieties of $\epsilon$-Fano type is bounded.
\end{conj}
The BAB Conjecture is one of the most important conjectures in birational geometry and it is related to the termination of flips. The following weaker conjecture is an important step towards the proof of the BAB Conjecture.
\begin{conj}[Weak BAB Conjecture]\label{conj2}
Fix $0<\epsilon<1$ and an integer $n>0$. 
Then there exists a  number $M(n,\epsilon)$ depending only
on $n$ and $\epsilon$ with the following property:
if  $X$ is an 
$n$-dimensional variety of $\epsilon$-Fano type, then 
$${\rm Vol}(-K_X) \leq M(n,\epsilon).$$
\end{conj}

The BAB Conjecture was proved in dimension two by Alexeev \cite{AK2} with a simplified argument by Alexeev and Mori \cite{AM}. In dimension three or higher, the BAB Conjecture is still open.\footnote{After this paper was posted on arXiv in 2014, Birkar \cite{Bir16a, Bir16b} gave a proof of the BAB Conjecture (along with the Weak BAB Conjecture and the Ambro's conjecture)  in arbitrary dimension in 2016 by different and much stronger methods.} There are only  some partial boundedness results. For example,  we have boundedness of   smooth Fano manifolds by Koll\'ar, Miyaoka, and Mori \cite{KMM92}, 
 of terminal $\mathbb{Q}$-Fano $\mathbb{Q}$-factorial threefolds of Picard number one
by Kawamata \cite{K},  of canonical $\mathbb{Q}$-Fano threefolds by  Koll\'ar, Miyaoka, Mori, and Takagi \cite{KMMT}, and of toric varieties by Borisov and Borisov \cite{BB}.

The Weak BAB Conjecture in dimension two was treated by Alexeev \cite{AK2}, Alexeev and Mori \cite{AM}, and Lai \cite{Lai}. Recently, the author \cite{J1} gave an optimal value for the number $M(2,\epsilon)$ (see also Corollary \ref{2bab}). For the Weak BAB Conjecture in dimension three assuming that the Picard number of $X$ is one,  an effective value of $M(3, \epsilon)$ was given by Lai \cite{Lai}. For the general case in dimension three and higher, the Weak BAB Conjecture is still open.

As the main theorem of this paper, we prove the Weak BAB Conjecture in dimension three.

\begin{thm}\label{BAB3main}
The Weak BAB Conjecture holds for $n=3$.
\end{thm}




{\it Acknowledgments.} The author would like to express his  
gratitude to his supervisor Professor Yujiro Kawamata, for  
suggestions, discussions, encouragement, and support.  
The author is indebted to Professors Florin Ambro, Caucher Birkar, Yoshinori Gongyo, Dr. Yusuke Nakamura,  and Mr. Pu Cao for effective conversations. 
A part of this paper  was done during the author's  visit to  University of Cambridge in  2013 and he would like to thank Professors   Caucher Birkar and Yifei Chen for the hospitality.

\section{Description of the proof}
Firstly, we give an approach to the Weak BAB Conjecture via Mori fiber spaces. 

\begin{definition}
A projective morphism $X\rightarrow T$  between normal varieties is called a \emph{Mori fiber space} if the following conditions hold:
\begin{itemize}
\item[(i)] $X$ is $\mathbb{Q}$-factorial with terminal singularities;

\item[(ii)] $f$ is a {\it contraction}, i.e., $f_*\OO_X=\OO_T$;

\item[(iii)] $-K_X$ is ample over $T$;

\item[(iv)] $\rho(X/T)=1$;

\item[(v)] $\dim X > \dim T$. 
\end{itemize}
In this case, we say that $X$ is endowed with a \emph{Mori fiber structure}. 
\end{definition}

We raise the following conjecture for Mori fiber spaces. 
\begin{conj}[Weak BAB Conjecture for Mori fiber spaces]
Fix $0<\epsilon<1$ and an integer $n>0$. 
Then there exists a  number $M(n,\epsilon)$ depending only
on $n$ and $\epsilon$ with the following property: if  $X$ is an
$n$-dimensional  variety  of $\epsilon$-Fano type with a Mori fiber structure, then 
$${\rm Vol}(-K_X) \leq M(n,\epsilon).$$ 
\end{conj}

This is just a special case of the Weak BAB Conjecture. But we can prove the following theorem by using Minimal Model Program.
\begin{thm}\label{main thm mfs}
The Weak BAB Conjecture holds for fixed $\epsilon$ and $n$ if and only if the Weak BAB Conjecture for Mori fiber spaces holds for fixed $\epsilon$ an $n$.
\end{thm}

By Theorem \ref{main thm mfs}, to bound the anti-canonical volumes of varieties of  Fano type, we only need to consider the ones with better singularities ($\bQ$-factorial terminal singularities) and with additional structures (Mori fiber structures). This is the advantage of this theorem. In dimension two, this  theorem appears as a crucial step to get the optimal value of $M(2,\epsilon)$ (c.f. \cite{J1}).

Restricting our interest to dimension three, we prove the following theorem.
\begin{thm}\label{mfs3}
The Weak BAB Conjecture for Mori fiber spaces holds for $n=3$.
\end{thm}

Theorem \ref{BAB3main} follows from Theorems \ref{main thm mfs} and \ref{mfs3} directly.

To prove Theorem  \ref{mfs3}, we need to consider  $3$-folds $X$ of $\epsilon$-Fano type with a Mori fiber structure $X\rightarrow T$. There are $3$ cases:
\begin{itemize}
\item[(1)] $\dim T=0$, $X$ is a  $\mathbb{Q}$-factorial terminal $\mathbb{Q}$-Fano  $3$-folds with $\rho=1$;
\item[(2)] $\dim T=1$, $X\rightarrow T\simeq\mathbb{P}^1$ is a {\it del Pezzo fibration}, i.e. a general fiber of which is a smooth del Pezzo surface;
\item[(3)] $\dim T=2$, $X\rightarrow T$ is a {\it conic bundle}, i.e. a general fiber of which is a smooth rational curve.
\end{itemize}

The second statement is implied by the following fact: if $(X, \Delta)$ is a klt log Fano pair, then $X$ is rationally connected (see \cite[Theorem 1]{ZQ}), in particular, for any surjective morphism $X\rightarrow T$ to a normal curve, $T\simeq \mathbb{P}^1$.

In Case (1), $X$ is bounded by Kawamata \cite{K}, and the optimal bound of ${\rm Vol}(-K_X)=(-K_X)^3$ is $64$ due to the classification on smooth Fano $3$-folds of Iskovskikh and Mori--Mukai and by Namikawa's   result  \cite{Nami} (Gorenstein case) and Prokhorov \cite{ProkK} (non-Gorenstein case).

We will mainly treat Cases (2) and (3). 

One basic idea is to construct singular pairs which are not klt along fibers of $X\rightarrow T$. Then by Connectedness Lemma, we may find a non-klt center intersecting with the fibers. Then restricting on a general fiber, we get the bound after some arguments on lower dimensional varieties.  But several difficulties  arise here.
 
In Case (3), the difficulty arises  in the construction of singular pairs
 because we need to avoid  components which are  vertical over $T$. To do this, we need a good understanding of the singularities and the boundedness of the surface $T$, which was done by several papers such as \cite{AK2, Birkar}.

In Case (2), the difficulty arises in the last step. After restricting to a general fiber, we need to bound the (unusual) log canonical thresholds on surfaces. So we are done by proving the (generalized) Ambro's conjecture in dimension two.

\begin{definition}
Let $(X, B)$ be  an lc pair and $D\geq 0$ be a $\bQ$-Cartier $\bQ$-divisor. The
{\it log canonical threshold} of $D$ with respect to $(X, B)$ is
$$\lct(X, B; D) = \sup\{t\in \bQ \mid (X, B+ tD) \text{ is lc}\}.$$
For the purpose of this paper, we need to consider the case when $D$ is not necessarily effective. 
Let $G$ be a $\bQ$-Cartier $\bQ$-divisor satisfying $G+B\geq 0$, The
{\it unusual log canonical threshold} of $G$ with respect to $(X, B)$ is
$$\ulct(X, B; G) = \sup\{t\in [0,1] \cap \bQ \mid (X, B+ tG) \text{ is lc}\}.$$
Note that the assumption $t\in [0,1]$ guarantees that $B+ tG\geq 0$.
\end{definition}
\begin{conj}[Ambro's conjecture]\label{ac}
Fix $0<\epsilon<1$ and an integer $n>0$. 
Then there exists a  number $\mu(n,\epsilon)>0$ depending only on $n$ and $\epsilon$ with the following property:
 if $(Y, B)$ is an $\epsilon$-klt log Fano pair  of dimension $n$, then
$$
\inf\{\lct(Y,B;D)\mid D\sim_\bQ-(K_Y+B), D\geq 0 \}\geq \mu(n,\epsilon).
$$
\end{conj}
Note that we do not assume any special conditions on the coefficients of $B$. The left-hand side of the inequality is called {\it  $\alpha$-invariant} of $(Y, B)$ which generalizes the concept of  $\alpha$-invariant of Tian for Fano manifolds in differential geometry (see \cite{CGM, Dem, Tian}). Recently Ambro \cite{Am} gave  a proof of this conjecture assuming that $(Y, B)$ is a toric pair where an explicit sharp number $\mu(n, \epsilon)$ was given.
For the purpose of this paper, we  need a stronger version of this conjecture where $D$ may not be effective.
\begin{conj}[generalized Ambro's conjecture]\label{gac}
Fix $0<\epsilon<1$ and an integer $n>0$. 
Then there exists a   number $\mu(n,\epsilon)>0$ depending only on $n$ and $\epsilon$ with the following property:
 if $(Y, B)$  is an  $\epsilon$-klt  weak log Fano pair  of dimension $n$  and $Y$ has  at worst terminal singularities, then
$$
\inf\{\ulct(Y,B;G)\mid G\sim_\bQ-(K_Y+B), G+B\geq 0 \}\geq \mu(n,\epsilon).
$$
\end{conj}
Note that Conjecture \ref{ac} follows from Conjecture \ref{gac} easily after taking a terminalization of $(Y,B)$.

We prove the conjecture in dimension two by following some ideas in the proof of the BAB Conjecture in dimension two  (\cite{AK2, AM}). But it seems that this conjecture does not follow from the BAB Conjecture  trivially.
 \begin{thm}\label{gac2}
Conjecture \ref{gac} holds for $n=2$.
\end{thm}


This paper is organized as follows. In Section \ref{section reduction}, we prove the reduction step to Mori fiber spaces (Theorem \ref{main thm mfs}). In Section \ref{section ambro}, we prove the generalized Ambro's conjecture in dimension two  (Theorem \ref{gac2}). In Section \ref{section mfs}, we prove the Weak BAB Conjecture for Mori fiber spaces in dimension three (Theorem \ref{mfs3}). 

\section{Preliminaries}

\subsection{Volumes}
\begin{definition}
Let $X$ be an $n$-dimensional projective variety  and $D$ be a Cartier divisor on $X$. The {\it volume} of $D$ is the real number
$$
{\rm Vol}(D)=\limsup_{m\rightarrow \infty}\frac{h^0(X,\OO_X(mD))}{m^n/n!}.
$$
Note that the limsup is actually a limit. Moreover by the homogenous property of the volume, we can extend the definition to $\bQ$-Cartier $\bQ$-divisors. Note that if $D$ is a nef $\bQ$-divisor, then $\Vol(D)=D^n$. If $D$ is a non-$\bQ$-Cartier $\bQ$-divisors, we may take a $\bQ$-factorialization of $X$, i.e., a birational morphism $\phi:Y\to X$ which is isomorphic in codimension one and $Y$ is $\bQ$-factorial, then $\Vol(D):=\Vol(\phi^{-1}_*D)$. Note that $\bQ$-factorialization always exists for klt pairs (cf. \cite[Theorem 1.4.3]{BCHM}).
\end{definition}

For more background on volumes, see \cite[11.4.A]{Positivity2}.

\subsection{Hirzebruch surfaces}
We
recall some basic properties of the Hirzebruch surfaces
$\mathbb{F}_n=\mathbb{P}_{\mathbb{P}^1}(\mathcal{O}_{\mathbb{P}^1}\oplus\mathcal{O}_{
\mathbb{P}^1}(n))$, $n\geq  0$. Denote by $h$ (resp. $f$) the
class in $\textrm{Pic }\mathbb{F}_n$ of the tautological bundle
$\mathcal{O}_{\mathbb{F}_n}(1)$ (resp. of a fiber). Then
$\textrm{Pic }\mathbb{F}_n=\mathbb{Z}h\oplus \mathbb{Z}f$ with
$f^2=0$, $f\cdot h=1$, $h^2=n$. If $n>0$, there is a unique
irreducible curve $\sigma_n\subset \mathbb{F}_n$ such that $\sigma_n \sim h-nf$, $\sigma_n^2=-n$. For
$n=0$, we can also choose one curve whose class in $\textrm{Pic
}\mathbb{F}_0$ is $h$ and denote it by $\sigma_0$.  
Note that 
$$
-K_{\mathbb{F}_n}\sim 2h-(n-2)f\sim 2\sigma_n+(n+2)f.
$$
\begin{lem}\label{multif}
For an effective $\bQ$-divisor $D\sim_\bQ  -K_{\mathbb{F}_n}$ and a   fiber $f$,  $\mult_f D\leq n+2$.
\end{lem}
\begin{proof}
Since $D-(\mult_f D)f$ is effective, $(D-(\mult_f D)f)\cdot h\geq 0$. On the other hand, $(D-(\mult_f D)f)\cdot h=n+2-\mult_f D$.
\end{proof}
\begin{lem}\label{multQ}
Let $T=\bP^2$ or $\mathbb{F}_n$,  then for an effective $\bQ$-divisor $D\sim_\bQ  -K_{T}$ and a point $Q$,
$\mult_Q D\leq n+4$ holds. Moreover, if we write $D=\sum_j b_jD_j$ where $D_j$ are distinct prime divisors and assume that $b_j\leq 1$ for all $j$, then $\sum_j b_j\leq 4.$
\end{lem}
\begin{proof}
If $T=\bP^2$, taking a general line $L$ through $Q$, we have
$$
3=(D\cdot L)\geq \mult_{Q}(D).
$$

If  $T=\mathbb{F}_n$, take $f$ to be the fiber passing through $Q$, by Lemma \ref{multif} and intersection theory, we have
$$
2=D\cdot f\geq \mult_{Q}D- \mult_{f}D\geq  \mult_{Q}D-n-2.
$$

For the latter statement, if $T=\mathbb{F}_n$, then the conclusion  follows from  \cite[Lemma 1.4]{AM}.  If $T=\mathbb{P}^2$, then $\sum b_j\leq 3$ by degree computation. 
\end{proof}

\subsection{Non-klt centers and connectedness lemma}
\begin{definition}
Let $X$ be a normal projective variety and $\Delta$ be a $\bQ$-divisor on $X$ such that $K_X+\Delta$ is $\bQ$-Cartier. Let $f: Y\rightarrow X$ be a log
resolution of $(X, \Delta)$, write
$$
K_Y =f^*(K_X+\Delta)+\sum a_iF_i,
$$
where $F_i$ is a prime divisor.  $F_i$ is called a {\it non-klt place}  of $(X, \Delta)$  if $a_i\leq -1$.
A proper subvariety $V\subset X$ is called a {\it non-klt center} of $(X, \Delta)$ if it is the image of a non-klt place. The {\it non-klt locus} $\text{Nklt}(X, \Delta)$ is the union of all  non-klt centers of $(X, \Delta)$. A non-klt center is {\it maximal} if it is an irreducible component of $\text{Nklt}(X, \Delta)$.
\end{definition}

The following lemma suggests a standard way to construct non-klt centers.
\begin{lem}[{cf. \cite[Lemma 2.29]{KM}}]\label{dimk}
Let $(X, \Delta)$ be a pair and $Z\subset X$ be a closed subvariety of codimesion $k$ such that $Z$ is not contained in the singular locus of $X$. If $\mult_Z \Delta\geq k$, then $Z$ is a non-klt center of $(X, \Delta)$. 
\end{lem}
Recall that the {\it multiplicity} $\mult_ZF$ of a divisor $F$ along a subvariety $Z$ is defined by the multiplicity $\mult_xF$ of $F$ at a general point $x\in Z$.

Unfortunately, the converse of Lemma \ref{dimk} is not true unless $k=1$. Usually we do not have good estimates for the multiplicity along a non-klt center except the following lemma.

\begin{lem}[{cf. \cite[Theorem 9.5.13]{Positivity2}}]
Let $(X, \Delta)$ be a pair and $Z\subset X$ be a non-klt center of $(X, \Delta)$ such that $Z$ is not contained in the singular locus of $X$. Then $\mult_Z \Delta\geq 1$. 
\end{lem}

If we assume some simple normal crossing condition on the boundary, we can get more information on the multiplicity along a non-klt center. For simplicity, we just consider surfaces.
\begin{lem}[{cf. \cite[4.1 Lemma]{McK}}]\label{multi}
Fix $0<e<1$. Let $S$ be a smooth surface, $B$ be an effective $\bQ$-divisor, and $D$ be a (not necessarily effective) simple normal crossing supported $\bQ$-divisor. Assume that the coefficients of $D$ are at most $e$ and $\mult_P B\leq 1-e$ for some point $P$, then for an arbitrary divisor $E$ centered on $P$ over $S$, $a_E(S, B+D)\geq -e$. 
In particular, if $Z$ is a non-klt center of $(S, B+D)$ and the coefficients of $D$ are at most $e$, then $\mult_Z B> 1-e$.
\end{lem}
\begin{proof}
By taking a sequence of point blow-ups, we can get the divisor $E$. Consider the blow-up at $P$, we have $f:S_1\rightarrow S$ with $$K_{S_1}+B_1+D_1+mE_1=f^*(K_S+B+D)$$ where $B_1$ and $D_1$ are the strict transforms of $B$ and $D$ respectively, and $E_1$ is the exceptional divisor with $$m=\mult_P(B+D)-1\leq 1-e+2e-1=e.$$ 
Now $D_1+mE_1$ is again simple normal crossing supported and $\mult_QB_1\leq \mult_PB$ for $Q\in E_1$. Hence, replacing $(S, B, D, P)$ by $(S_1, B_1, D_1+mE_1, Q)$,  by induction on the minimal number of blow-ups, we conclude that  the coefficient of $E$ is at most $e$ and hence $a_{E}(S, B+D)\geq -e$. 
\end{proof}

We have the following connectedness lemma of Koll\'{a}r and Shokurov for non-klt locus (cf.  
Shokurov \cite{Shokurov}, Koll\'{a}r \cite[17.4]{Kol92}).
\begin{thm}[Connectedness Lemma]
Let $f:X\rightarrow Z$ be a proper morphism of normal varieties with connected fibers and $D$ is a $\bQ$-divisor such that $-(K_X+D)$ is $\bQ$-Cartier, $f$-nef, and $f$-big. Write $D=D^+-D^-$ where $D^+$ and $D^-$ are effective with no common components. If $D^-$ is $f$-exceptional (i.e. all of its components have image of codimension at least $2$), then ${\rm Nklt} (X,D)\cap f^{-1}(z)$ is connected for any $z\in Z$. 
\end{thm}
\begin{remark}
There are two main cases of interest in the Connectedness Lemma:
\begin{itemize}
\item[(i)] $Z$ is a point and $(X,D)$ is a weak log Fano pair. Then $\Nklt(X,D)$ is connected.
\item[(ii)] $f:X\rightarrow Z$ is birational, $(Z,B)$ is a log pair and $K_X+D=f^*(K_Z+B)$.
\end{itemize}
\end{remark}

\section{Reduction to Mori fiber spaces}\label{section reduction}
In this section, we prove the reduction step to Mori fiber spaces (Theorem \ref{main thm mfs}). The ``only if''  direction is trivial,  the ``if''  direction is a consequence of the following theorem.

\begin{thm}\label{thm MFS}Fix an integer $n>0$ and $0<\epsilon<1$. 
Every $n$-dimensional variety $X$ of $\epsilon$-Fano type is birational to an $n$-dimensional variety  $X'$ of $\epsilon$-Fano type  with a Mori fibration structure such that $$\Vol(-K_X)\leq \Vol(-K_{X'}).$$
\end{thm}

We start with two lemmas. The first lemma is about equivalent definitions of $\epsilon$-Fano type.

\begin{lem}[{cf. \cite[Lemma-Definition 2.6]{PS}}]\label{FT=CY}
Let $Y$ be a projective  normal variety, and $\epsilon\in [0,1)$. The following are equivalent:
\begin{enumerate}
\item $Y$ is of $\epsilon$-Fano type;

\item There exists an effective $\mathbb{Q}$-divisor  $\Delta$ such that $\Delta$ is big, $(Y, \Delta)$ is $\epsilon$-klt, and $K_Y+\Delta\equiv 0$.
\end{enumerate}
\end{lem}
\begin{proof}
First we assume that  $Y$ is of $\epsilon$-Fano type, that is, there exists an effective $\mathbb{Q}$-divisor $B$ on $Y$ such that $(Y, B)$ is $\epsilon$-klt log Fano pair. Then take a general effective ample $\mathbb{Q}$-divisor $A$ on $Y$ such that $(Y, B+A)$ is $\epsilon$-klt and 
$$
K_Y+B+A\sim_{\mathbb{Q}}0.
$$
We may take $\Delta=A+B$.

Then we assume that there exists an effective $\mathbb{Q}$-divisor  $\Delta$ such that $\Delta$ is big, $(Y, \Delta)$ is $\epsilon$-klt, and $K_Y+\Delta\equiv 0$. Since $\Delta$ is big, we may write $\Delta=A+G$ where $A$ is an ample $\mathbb{Q}$-divisor and $G$ is an effective $\mathbb{Q}$-divisor. We may take a sufficiently small $\delta>0$ such that $(Y, \Delta+\delta G)$ is again $\epsilon$-klt. Hence $(Y, (1-\delta)\Delta+\delta G)$ is $\epsilon$-klt, and 
$$-(K_Y+ (1-\delta)\Delta+\delta G)\equiv \delta A$$
is ample. Hence $Y$ is of $\epsilon$-Fano type.
\end{proof}

Being of $\epsilon$-Fano type is preserved by MMP according to the following lemma.
\begin{lem}[{cf. \cite[Lemma 3.1]{GOST}}]\label{mmp lem}
Let $Y$ be a projective  normal variety and $f: Y\rightarrow Z$ be a projective birational morphism.  
\begin{enumerate}
\item If $Y$ is of $\epsilon$-Fano type, then so is $Z$;

\item Assume that $f$ is small, then $Y$ is of $\epsilon$-Fano type if and only if so is $Z$.
\end{enumerate}
In particular, minimal model program preserves $\epsilon$-Fano type.
\end{lem}

\begin{proof}
First we assume that  $Y$ is of $\epsilon$-Fano type, that is, by Lemma \ref{FT=CY}, there exists an effective $\mathbb{Q}$-divisor  $\Delta$ such that $\Delta$ is big, $(Y, \Delta)$ is $\epsilon$-klt, and $K_Y+\Delta\equiv 0$. 
Pushing forward by $f$, by negativity lemma, 
$$
K_Y+\Delta=f^*(K_Z+f_*\Delta)\equiv 0.
$$
Hecne $f_*\Delta$ is big, $(Z,f_*\Delta)$ is $\epsilon$-klt, and $K_Z+f_*\Delta\equiv 0$, that is, $Z$ is of $\epsilon$-Fano type.

Next we assume that  $f$ is small and  $Z$ is of $\epsilon$-Fano type. Let $\Gamma$ be an effective big $\mathbb{Q}$-divisor on $Z$ such that $(Z, \Gamma)$ is $\epsilon$-klt and $K_Z+\Gamma\equiv 0$. Let $\Delta$ be the strict transform of $\Gamma$ on $Y$. Then $\Delta$ is big since $f$ is small. Again by $f$ is small, 
$$
K_Y+\Delta=f^*(K_Z+\Gamma).
$$
Hence $(Y, \Delta)$ is $\epsilon$-klt and $K_Y+\Delta\equiv 0$. Hence $Y$ is of $\epsilon$-Fano type. 
\end{proof}
\begin{proof}[Proof of Theorem \ref{thm MFS}]
Fix $0<\epsilon<1$ and an integer $n>0$. Let  $X$ be a variety of $\epsilon$-Fano type of dimension $n$, that is, by Lemma \ref{FT=CY}, there exists an effective $\mathbb{Q}$-divisor  $\Delta$ such that $\Delta$ is big, $(X, \Delta)$ is $\epsilon$-klt, and $K_X+\Delta\equiv 0$. 
By \cite[Corollary 1.4.3]{BCHM},  taking a $\bQ$-factorialization of $(X, \Delta)$, we have a birational morphism $\phi: X_0 \rightarrow X$ where
$K_{X_0}+\phi^{-1}_*\Delta=\phi^*(K_X+\Delta)$, $X_0$ is $\bQ$-factorial, and $\phi$ is isomorphic in codimension one. 

Again by \cite[Corollary 1.4.3]{BCHM},  taking a terminalization of $X_0$, 
we have a birational morphism $\pi: X_1 \rightarrow X_0$ where
$K_{X_1}+\Delta_{X_1}=\pi^*(K_{X_0}+\phi^{-1}_*\Delta)$, 
$\Delta_{X_1}\geq \pi^*(\phi^{-1}_*\Delta)$ is an effective $\bQ$-divisor, $X_1$ is $\bQ$-factorial and terminal.
Here $K_{X_1}+\Delta_{X_1}\equiv 0$ and $(X_1,\Delta_{X_1})$ is  $\epsilon$-klt. Since $\Delta$ is big and $\phi$ is small, $\Delta_{X_1}\geq \pi^*(\phi^{-1}_*\Delta)$ is big. Therefore, $X_1$ is $\bQ$-factorial terminal and of $\epsilon$-Fano type.

Running the $K$-MMP on $X_1$, we get a sequence of normal projective varieties: 
$$
X_1\dashrightarrow X_2 \dashrightarrow X_3\dashrightarrow \cdots\dashrightarrow X_r\rightarrow T. 
$$ 
Since $-K_{X_1}$ is big, this sequence ends with a Mori fiber space $X_r\rightarrow T$ (cf. \cite[Corollary 1.3.3]{BCHM}). Since we run the $K$-MMP, $X_r$ is again $\bQ$-factorial and terminal. 
By Lemma \ref{mmp lem}, for all $i$, $X_i$ is of $\epsilon$-Fano type. Now $X_r$ is an $n$-dimensional variety of  $\epsilon$-Fano type with a Mori fiber structure by construction, which is birational to $X$. 

The remaining thing is to compare the anti-canonical volumes. By definition, $$\Vol(-K_X)=\Vol(-K_{X_0}).$$ Since $X_1$ is a terminalization of $X_0$, we have $K_{X_1}+F=\pi^*K_{X_0}$ with $F$ an effective $\bQ$-divisor.  Hence $$\Vol(-K_{X_0})=\Vol(-(K_{X_1}+F))\leq \Vol(-K_{X_1}).$$
By Lemma \ref{vol lem} below, 
$$
\Vol(-K_{X_1})\leq \Vol(-K_{X_r}).
$$

Hence we may take $X'=X_r$ and complete the proof of Theorem \ref{thm MFS}. 
\end{proof}

\begin{lem}\label{vol lem}
Let $X_i\dashrightarrow X_{i+1}$ be one step of the $K$-MMP. Then
$$
{\rm Vol}(-K_{X_i})\leq {\rm Vol}(-K_{X_{i+1}}).
$$
\end{lem}

\begin{proof}
Take a common resolution $p:W \rightarrow X_i$, $q: W\rightarrow X_{i+1}$. Then
$$
p^*(K_{X_i})=q^*(K_{X_{i+1}})+E,
$$
where $E$ is an effective $q$-exceptional $\mathbb{Q}$-divisor. 
Hence
\begin{align*}
 {\rm Vol}(-K_{X_i}) 
={}&{\rm Vol}(-p^*(K_{X_i}))\\
={}&{\rm Vol}(-q^*(K_{X_{i+1}})-E)\\
\leq {}&{\rm Vol}(-q^*(K_{X_{i+1}}))\\
= {}&{\rm Vol}(-K_{X_{i+1}}).
\end{align*}
We proved the lemma.
\end{proof}

As a direct corollary, we recover the main result in \cite{J1} on the Weak BAB Conjecture in dimension two. 
\begin{cor}\label{2bab}
Fix $0<\epsilon<1$. 
Then there exists a  number 
$$M(2,\epsilon):=\max\Big\{9, \rounddown{2/\epsilon}+4+\frac{4}{\rounddown{2/\epsilon}}\Big\}$$
 with the following property: if  $X$ is a surface of $\epsilon$-klt  del Pezzo type, then 
$${\rm Vol}(-K_X) \leq M(2,\epsilon).$$
\end{cor}
\begin{proof}
By Theorem \ref{thm MFS}, we only need to consider the cases when $X=\mathbb{P}^2$ or $\mathbb{F}_n$ with $n\leq  2/\epsilon$ (see \cite[Lemma 1.4]{AM} or \cite[Lemma 3.1]{J1}). And the result follows from an easy computation of volumes. 
\end{proof}

\section{Generalized Ambro's conjecture in dimension two}\label{section ambro}
In this section, we prove the generalized the Ambro's conjecture in dimension two (Theorem \ref{gac2}).

Fix an $\epsilon$-klt weak log del Pezzo pair $(S,B)$ with $S$ smooth and a $\bQ$-divisor $G\sim_\bQ -(K_S+B)$ such that $G+B\geq 0$. Set $a:=\ulct(S,B;G)$. Since we work with $\bQ$-divisors, $a$ is a positive rational number. The problem is to bound $a$ from below. We may assume that $a<1$. Set $D=G+B\geq 0$. Then $(S, B+aG)=(S,(1-a)B+aD)$ is not klt.
Note that $D\sim_\bQ -K_S$.

By Base Point Free Theorem (cf. \cite[Theorem 3.3]{KM}),
$-(K_S+B)$ is semi-ample. Hence there exists an effective
$\bQ$-divisor $M$  such that
$K_S+B+M\sim_\bQ 0$ and $(S, B+M)$ is $\epsilon$-klt. 
For any birational morphism $f:S\rightarrow T$ between smooth surfaces, we have
\begin{align*}
K_S+B+M={}&f^*(K_T+f_*B+f_*M), \\ 
K_S+(1-a)(B+M)+aD={}&f^*(K_T+(1-a)(f_*B+f_*M)+af_*D).
\end{align*}
Hence $(T, f_*B+f_*M)$ is  $\epsilon$-klt and   $(T, (1-a)(f_*B+f_*M)+af_*D)$ is not klt with 
$$
K_T+f_*B+f_*M\sim_\bQ K_T+(1-a)(f_*B+f_*M)+af_*D\sim_\bQ 0.
$$
Recall that either $S\simeq \mathbb{P}^2$ or there exists a birational
morphism $g:S\rightarrow \mathbb{F}_n$ with $n\leq  2/\epsilon$ by \cite[Lemma 1.4]{AM} or \cite[Lemma 3.1]{J1}. 

Hence by replacing $S$ by $T=\bP^2$ or $\mathbb{F}_n$, we may assume  that there exists a triple $(T, B_T, D_T)$ satisfying the following conditions:
\begin{itemize}
\item[(i)] $T=\bP^2$ or $\mathbb{F}_n$ with $n\leq  2/\epsilon$;
\item[(ii)] $B_T, D_T$ are effective $\bQ$-divisors on $T$;
\item[(iii)] $(T, B_T)$ is $\epsilon$-klt and   $(T, (1-a)B_T+aD_T)$ is not klt; 
\item[(iv)] $K_T+B_T\sim_\bQ K_T+(1-a)B_T+aD_T\sim_\bQ 0$, equivalently, $B_T\sim_\bQ  D_T\sim_\bQ -K_T$.
\end{itemize}
Since $(T, (1-a)B_T+aD_T)$ is not klt, we may take a sequence of point blow-ups 
$$
T_{r+1}\rightarrow T_{r}\rightarrow \cdots \rightarrow T_{2}\rightarrow T_{1}=T
$$
where $T_{i+1}\rightarrow T_i$ is the blow-up at a non-klt center $P_i\in \Nklt(T_i, (1-a)B_i+aD_i+E_i)$ where $B_i$ and   $D_i$
are the strict transforms of $B_T$ and   $D_T$ respectively and 
$$
K_{T_i}+(1-a)B_i+aD_i+E_i=\pi_i^*(K_T+(1-a)B_T+aD_T),
$$
where $\pi_i:T_i\rightarrow T$ is the composition map and $E_i$ is a $\pi_i$-exceptional $\bQ$-divisor. 
We stop this process at $T_{r+1}$ if 
$$\dim\Nklt(T_{r+1}, (1-a)B_{r+1}+aD_{r+1}+E_{r+1})>0.$$
Since $P_i$ is a non-klt center of $(T_i, (1-a)B_i+aD_i+E_i)$, $\mult_{P_i}((1-a)B_i+aD_i+E_i)\geq 1$. Note that the coefficients of $E_i$ are $(\mult_{P_j}((1-a)B_j+aD_j+E_j)-1)$ for $j<i$, hence $E_i$ is effective for all $i$. Furthermore, we may assume that $\mult_{P_i}B_i$ is non-increasing. 
Take the integer $k\leq r$ such that $\mult_{P_i}B_i\geq \epsilon/2$ for $i\leq k$ and $\mult_{P_i}B_i<  \epsilon/2$ for $i> k$.
Write $B_T=\sum_j b_jB^j$ where $B^j$ are distinct prime divisors and $B_i=\sum_j b_jB_i^j$ where $B_i^j$ are the strict transforms of $B^j$. We have $b_j<1-\epsilon$ since  $(T, B_T)$ is $\epsilon$-klt. Recall that $\sum_j b_j\leq 4$ by Lemma \ref{multQ}.

\begin{claim}
If $\mult_{B^j}(aD_T)>\epsilon/2$ for some $j$, then $a\geq  \epsilon^2/({4+4\epsilon})$.
\end{claim}
\begin{proof}
Recall that $T=\bP^2$ or $\mathbb{F}_n$ with $n\leq  2/\epsilon$.

If $T=\bP^2$, then  $\mult_{B^j}D_T\leq 3$ by degree counting.
If $T= \mathbb{F}_n$ and $B^j$ is a fiber, then $\mult_{B^j}D_T\leq n+2\leq  2/{\epsilon}+2$ by Lemma \ref{multif}. 
If $T= \mathbb{F}_n$ and $B^j$ is not a fiber, then $\mult_{B^j}D_T\leq D_T\cdot f=2$ where $f$ is a fiber.
Hence 
$$
a\geq  \frac{\epsilon}{2\mult_{B^j}D_T}\geq \frac{\epsilon^2}{4+4\epsilon}.
$$
We proved the claim.
\end{proof}

Since we need a lower bound of $a$, from now on, by Claim 1, we may assume that  $\mult_{B^j}(aD_T)\leq \epsilon/{2}$ for all $j$. In particular, $\mult_{B_i^j}(aD_i)\leq  \epsilon/2$ and  $$\mult_{B_i^j}((1-a)B_i+aD_i) <  1-\epsilon/2$$   for all $i$ and $j$.

\begin{claim}
$(B_{k+1}^j)^2\geq -{4}/{\epsilon}$ for all $j$.
\end{claim}
\begin{proof}
If $(B_{k+1}^j)^2<0$, then
\begin{align*}
-2\leq{}&
2p_a(B_{k+1}^j)-2=(K_{T_{k+1}}+B_{k+1}^j)\cdot B_{k+1}^j\\
={}&\frac{\epsilon}{2} (B_{k+1}^j)^2+\Big(K_{T_{k+1}}+\Big(1-\frac{\epsilon}{2}\Big)B_{k+1}^j\Big)\cdot B_{k+1}^j\\
\leq{}&\frac{\epsilon}{2} (B_{k+1}^j)^2+(K_{T_{k+1}}+ (1-a)B_{k+1}+aD_{k+1}+E_{k+1})\cdot B_{k+1}^j\\
={}&\frac{\epsilon}{2} (B_{k+1}^j)^2<0,
\end{align*}
Hence we proved the claim.
\end{proof}

Now we can bound the number $k$.  On $T_{k+1}$, we have
\begin{align*}
(B_{k+1})^2={}&(\sum_j b_jB_{k+1}^j)^2\geq \sum_j b_j^2(B_{k+1}^j)^2\geq (\sum_j b_j^2)(-4/\epsilon)\\
\geq{}&  (\sum_j b_j)(1-\epsilon)(-4/\epsilon)\geq 16-\frac{16}{\epsilon}
\end{align*}
and $(B_1)^2=(K_T)^2\leq 9$.
On the other hand, after each blow-up, $(B_i)^2$ decreases by at least $\epsilon^2/4$ by the assumption $\mult_{P_i}B_i\geq \epsilon/2$ for $i\leq k$. Hence 
$$
{k}\leq \frac{9-(16-16/\epsilon)}{\epsilon^2/4}\leq \frac{64}{\epsilon^3}.
$$

Now we consider $\pi_{k+1}^*(aD_T)$ on $T_{k+1}$. 
\begin{claim}\label{QQQ}There exists a point $Q$ on $T_{k+1}$ such that  $\mult_Q\pi_{k+1}^*(aD_T)\geq \epsilon/4$.
\end{claim}
\begin{proof}
Consider the pair $(T_{k+1}, (1-a)B_{k+1}+aD_{k+1}+E_{k+1})$.  Note that $E_{k+1}$ is simple normal crossing supported.

Suppose that there exists a curve $E$ with coefficient at least $1-3\epsilon/4$ in $E_{k+1}$, that is, $$\mult_E(K_{T_{k+1}}-\pi_{k+1}^*(K_T+(1-a)B_T+aD_T))\leq -1+3\epsilon/4.$$
 On the other hand, since $(T, (1-a)B_T)$ is $\epsilon$-klt, 
$$\mult_E(K_{T_{k+1}}-\pi_{k+1}^*(K_T+(1-a)B_T))> -1+\epsilon.$$
Hence $\mult_E\pi_{k+1}^*(aD_T)\geq  {\epsilon}/{4}$.

Then we may assume that all coefficients of $E_{k+1}$ are smaller than $1-3\epsilon/4$,  then $k<r$ and $P_{k+1}$ is a non-klt center of  $(T_{k+1}, (1-a)B_{k+1}+aD_{k+1}+E_{k+1})$. 
By Lemma \ref{multi}, $\mult_{P_{k+1}}((1-a)B_{k+1}+aD_{k+1}) \geq 3\epsilon/4$. Then $\mult_{P_{k+1}}(aD_{k+1}) \geq \epsilon/4$  since $\mult_{P_{k+1}}B_{k+1} < \epsilon/2$ by definition of $k$. In particular, $\mult_{P_{k+1}}\pi_{k+1}^*(aD_T)\geq \mult_{P_{k+1}}(aD_{k+1}) \geq \epsilon/4$.

We proved the claim.
\end{proof}
Now we will estimate $\mult_{Q_1} (aD_T)$ where $Q_1$ is the image of $Q$ on $T$. By removing unnecessary blow-ups, we may assume that we have a sequence of blow-ups
$$
T_{k+1}\rightarrow T_{k}\rightarrow \cdots \rightarrow T_{2}\rightarrow T_{1}=T
$$
where $f_{i+1}: T_{i+1}\rightarrow T_i$ is the blow-up at  $Q_i$ which is the image of $Q$ on $T_i$ with $k\leq 64/\epsilon^3$. Recall that $\pi_i:T_i\rightarrow T$ is the composition map and $D_i$ is the strict transform of $D_T$ on $T_i$. Denote $C_{i+1}$ to be the exceptional divisor of $f_{i+1}$ and  $C_{i+1}^j$ be its strict transform on $T_{j}$ for $j\geq i+1$. We can write
$$
\pi_j^*(aD_T)=aD_j+\sum_{2\leq i\leq j}c_iC_i^j,
$$
with $c_i=\mult_{Q_{i-1}} \pi_{i-1}^*(aD_T)$.
\begin{claim}\label{FFF}
If $\mult_{Q_1} (aD_T)\leq \alpha$, then $\mult_{Q_i} \pi_i^*(aD_T)\leq  ({\mathsf F}_{i+1}-1)\alpha$ for $1\leq i\leq k+1$. Here ${\mathsf F}_n$ is the Fibonacci number with relation ${\mathsf F}_n={\mathsf F}_{n-1}+{\mathsf F}_{n-2}$ for all $n\geq 2$ and  ${\mathsf F}_0= {\mathsf F}_1=1$.
\end{claim}
\begin{proof}
We run induction on $i$. The case $i=1$ is trivial. Assume the conclusion holds for $i<j$, then noting that
$Q_j$ is contained in at most two exceptional curves, we have
\begin{align*}
\mult_{Q_j} \pi_j^*(aD_T)={}&\mult_{Q_j}(aD_j+\sum_{2\leq i\leq j}c_iC_i^j)\\
\leq {}& \mult_{Q_j}(aD_j)+({\mathsf F}_{j}-1)\alpha+({\mathsf F}_{j-1}-1)\alpha\\
\leq {}& \mult_{Q_1}(aD_1)+({\mathsf F}_{j}-1)\alpha+({\mathsf F}_{j-1}-1)\alpha\\
\leq {}& ({\mathsf F}_{j+1}-1)\alpha.
\end{align*}
We proved the claim.
\end{proof}
By Claims \ref{QQQ} and \ref{FFF}, $\mult_{Q_1}(aD_T)\geq \epsilon/(4{\mathsf F}_{k+2}-4)$.
Recall that $D_T\sim_\bQ -K_T$ and $T=\bP^2$ or $\mathbb{F}_n$ with $n\leq 2/\epsilon$. 
By Lemma \ref{multQ}, $\mult_{Q_1}(D_T)\leq n+4$, combining with the inequality $k\leq 64/\epsilon^3$, we have
$$
a \geq  \frac{\epsilon^2}{(2+4\epsilon)(4{\mathsf F}_{\rounddown{64/\epsilon^3}+2}-4)},
$$
and hence we may take this number to be $\mu(2,\epsilon)$. 

We have proved Theorem \ref{gac2}.

\section{Weak BAB  Conjecture for Mori fiber spaces in dimension three}\label{section mfs}
In this section, we prove the Weak BAB Conjecture for Mori fiber spaces in dimension $3$ (Theorem \ref{mfs3}).
Recall that by a Mori fiber space we always mean a $\bQ$-factorial terminal one.

Fix $0<\epsilon<1$ and consider an $\epsilon$-klt log  Fano pair $(X,\Delta)$ of dimension $3$ with a Mori fiber structure. As explained, there are three cases:
\begin{itemize}
\item[(1)]  $X$ is a  $\mathbb{Q}$-factorial terminal $\mathbb{Q}$-Fano  $3$-folds with $\rho=1$;
\item[(2)]  $X\rightarrow  \mathbb{P}^1$ is a  del Pezzo fibration;
\item[(3)]  $X\rightarrow S$ is a {conic bundle}.
\end{itemize}
As mentioned before, Case (1) is done by Kawamata \cite{K}. We treat Cases (2) and (3) in the following two subsections, see Corollary \ref{cor dp} and Theorem \ref{thm conic}.
\subsection{Contractions to a curve}\label{curve section}
In this subsection, we treat the case under a more general setting when there is a contraction $f:X\rightarrow \bP^1$ (e.g. $X$ has a del Pezzo fibration structure).  
\begin{thm}
Let $(Y, B)$ be an $\epsilon$-klt log Fano pair of dimension $n$ with a contraction $g:Y\rightarrow \bP^1$ and $Y$ having terminal singularities. Assume that the Weak BAB Conjecture  and the generalized Ambro's conjecture hold in dimension $n-1$ with $M(n-1,\epsilon)$ and $\mu(n-1,\epsilon)$ the numbers defined in these conjectures. Then $$\Vol(-K_Y)\leq \frac{2nM(n-1,\epsilon)}{\mu(n-1,\epsilon)}.$$ 
\end{thm}
\begin{proof}
Note that $Y$ is  terminal by assumption.  Hence a general fiber $F$ of $g$ is terminal and of $\epsilon$-Fano type of dimension $n-1$ by adjunction formula. It follows that $\Vol(-K_F)\leq M(n-1,\epsilon)$ by the Weak BAB Conjecture  in dimension $n-1$.

 By contrary, assume that $$\Vol(-K_Y)> \frac{2nM(n-1,\epsilon)}{\mu(n-1,\epsilon)}.$$ 
Take a rational number $s$ satisfying  
$$\Vol(-K_Y)>s\cdot n M(n-1,\epsilon)> \frac{2nM(n-1,\epsilon)}{\mu(n-1,\epsilon)}.$$
Here we note that $s>2$ since $\mu(n-1,\epsilon)<1$ (by considering, for example, $(\bP^{n-1}, 0)$ in the Ambro's conjecture).

The following lemma allows us to construct non-klt centers. 
\begin{lem}\label{lemma1}
For a general fiber $F$ of $g$,  $-K_Y-sF$ is $\bQ$-effective.
\end{lem}
\begin{proof}
For a positive integer $p$ and a sufficiently divisible positive integer $m$, we have an exact sequence
$$
0\rightarrow \OO_Y(-mK_Y-pF)\rightarrow \OO_Y(-mK_Y-(p-1)F)\rightarrow \OO_F(-mK_Y-(p-1)F)\rightarrow 0.
$$
Note that  $\OO_F(-mK_Y-(p-1)F)= \OO_F(-mK_F)$. Hence
\begin{align*}
{}&h^0(Y, \OO_Y(-mK_Y-pF))\\
\geq {}&h^0(Y, \OO_Y(-mK_Y-(p-1)F))- h^0(F, \OO_F(-mK_F)).
\end{align*}
Inductively, we have
\begin{align*}
{}& h^0(Y, \OO_Y(-mK_Y-pF))\\
\geq {}& h^0(Y, \OO_Y(-mK_Y))- p\cdot h^0(F, \OO_F(-mK_F)).
\end{align*}
Now we may take the integer $p=sm$ since $m$ is sufficiently divisible. By the definition of volume, we have
\begin{align*}
{}&\lim_{m\rightarrow \infty}\frac{n!}{m^n}\bigg(h^0(Y, \OO_Y(-mK_Y))- sm\cdot h^0(F, \OO_F(-mK_F))\bigg)\\
={}&\Vol(-K_Y)-sn\Vol(-K_F)\\
\geq {}&\Vol(-K_Y)-snM(n-1, \epsilon)>0.
\end{align*}
Hence $h^0(Y, \OO_Y(-mK_Y-smF))>0$ for $m$ sufficiently divisible, that is, $-K_Y-sF$ is $\bQ$-effective. \end{proof}

By Lemma \ref{lemma1}, there is an effective $\bQ$-divisor $B'\sim_\bQ -\frac{1}{s}K_Y-F$.
Now for two general fibers $F_1$ and $F_2$ of $g$, consider the pair
$$\Big(Y, \Big(1-\frac{2}{s}\Big)B+2B'+F_1+F_2\Big).$$ 
By construction, $F_1\cup F_2\subset \Nklt(Y, (1-\frac{2}{s})B+2B'+F_1+F_2)$. Note that 
$$
-\Big(K_Y+\Big(1-\frac{2}{s}\Big)B+2B'+F_1+F_2\Big)\sim_\bQ -\Big(1-\frac{2}{s}\Big)(K_Y+B)
$$
is ample, since $s>2$. By Connectedness Lemma, $\Nklt(Y, (1-\frac{2}{s})B+2B'+F_1+F_2)$ is connected. 
Hence there is a non-klt center $W$ connecting $F_1$ and $F_2$. In particular, $W$ dominates $\bP^1$. Restricting on a general fiber $F$, by adjunction, we have
$(F, B|_F)$ is $\epsilon$-klt log Fano with $F$ terminal and $(F, (1-\frac{2}{s})B|_F+2B'|_F)$ is not klt (see \cite[Lemma 5.17, Lemma 5.50]{KM}) with $B'|_F\sim_\bQ -\frac{1}{s}K_F$. Hence 
$$\frac{2}{s}\geq \ulct(F, B|_F; sB'|_F-B|_F).$$
The generalized Ambro's conjecture comes in to play when trying to bound $s$ from above. By the generalized Ambro's conjecture in dimension $n-1$, 
$$
s\leq \frac{2}{\mu(n-1,\epsilon)},
$$
which contradicts the definition of $s$.
\end{proof}

In particular, by Corollary \ref{2bab} and Theorem \ref{gac2}, the Weak BAB Conjecture   and the generalized Ambro's conjecture hold in dimension $2$,  and hence the following corollary holds.

\begin{cor}\label{cor dp}
Let $X$ be a $3$-fold of $\epsilon$-Fano type with a contraction $f:X\rightarrow \bP^1$ and $X$ having terminal singularities. Then $$\Vol(-K_X)\leq \frac{6M(2,\epsilon)}{\mu(2,\epsilon)}.$$
\end{cor}

\subsection{Conic bundles}
In this subsection, we treat the case that $X$ has a conic bundle structure $f:X\rightarrow S$. Firstly we collect some facts about singularities of the surface $S$.
\begin{thm}\label{thmS}
Let $(X,\Delta)$ be an $\epsilon$-klt log Fano pair of dimension $3$ and $f:X\rightarrow S$ be a Mori fiber space to a surface $S$, then 
\begin{itemize}
\item[(i)] There exists an effective  $\bQ$-divisor $\Delta_S$ on $S$, such that $(S, \Delta_S)$ is klt log del Pezzo;
\item[(ii)] There exists an effective  $\bQ$-divisor $\Delta_S$ on $S$, such that $(S, \Delta_S)$ is $\delta(\epsilon)$-klt and $K_S+\Delta_S\sim_\bQ 0$, where $\delta(\epsilon)$  is a positive real number depending only on $\epsilon$;
\item[(iii)] The family of such $S$ is bounded.
\item[(iv)] There exists a positive  integer $d(\epsilon)$ depending only on $\epsilon$, such that on $S$ there is a very ample divisor $H$ satisfying $H^2\leq d(\epsilon)$. 
\end{itemize}
\end{thm}
\begin{proof}
(i) is by \cite[Corollary 3.3]{FG}. 
(ii) is by \cite[Corollary 1.7]{Birkar} since we may find a boundary $\Delta'\geq \Delta$ such that $(X,\Delta')$ is $\epsilon$-klt and $K_X+\Delta'\sim_\bQ 0$. (iii) is by (ii) and \cite[Theorem 6.8]{AK2}. (iv) is a direct consequence of (iii).
\end{proof}

If taking $H$ generally, we may assume that $G:=f^{-1}(H)$ and $H$ are smooth. Note that  $f|_G: G\to H$ is a conic bundle from a smooth surface to a smooth curve. Note that for a general fiber $F$ of $f$ (or $f|_G$), $G|_G\sim_\bQ d(\epsilon) F$.
\begin{lem}\label{lemma2}$\Vol(-K_X|_G)\leq \frac{8(d(\epsilon)+2)}{\epsilon}$.
\end{lem}
\begin{proof}
Assume by contrary that $\Vol(-K_X|_G)> \frac{8(d(\epsilon)+2)}{\epsilon}$. Choose a rational number $r$ such that
$$\Vol(-K_X|_G)> 4r>\frac{8(d(\epsilon)+2)}{\epsilon}.$$
For a positive integer $p$ and a sufficiently divisible positive integer $m$, for a general fiber $F$ of $f|_G$, we have an exact sequence
$$
0\rightarrow \OO_G(-mK_X|_G-pF)\rightarrow \OO_G(-mK_X|_G-(p-1)F)\rightarrow \OO_F(-mK_F)\rightarrow 0.
$$
 Hence
\begin{align*}
{}&h^0(G, \OO_G(-mK_X|_G-pF))\\
\geq{}& h^0(G, \OO_G(-mK_X|_G-(p-1)F))- h^0(F, -mK_F).
\end{align*}
Inductively, we have
$$
h^0(G, \OO_G(-mK_X|_G-pF))\geq h^0(G, \OO_G(-mK_X|_G))- p\cdot h^0(F, -mK_F).
$$
Now we may take the integer $p=rm$ since $m$ is sufficiently divisible. By the definition of volume, we have
\begin{align*}
{}&\lim_{m\rightarrow \infty}\frac{2}{m^2}\bigg( h^0(G, \OO_G(-mK_X|_G))- rm\cdot h^0(F, -mK_F)\bigg)\\
={}&\Vol(-K_X|_G)-2r\Vol(-K_F)\\
={}&\Vol(-K_X|_G)-4r>0.
\end{align*}
Here note that $F\simeq \bP^1$.
Hence $h^0(G, \OO_G(-mK_X|_G-rmF))>0$ for $m$ sufficiently divisible, that is, $-K_X|_G-rF$ is $\bQ$-effective. Take an effective $\bQ$-divisor $D\sim_\bQ-K_X|_G-rF$ on $G$. For two general fibers $F_1$ and $F_2$ of $f|_G$,  consider the pair 
$$\Big(G, \Big(1-\frac{d(\epsilon)+2}{r}\Big)\Delta|_G+\frac{d(\epsilon)+2}{r}D+F_1+F_2\Big).$$
It is log Fano since 
\begin{align*}
{}&-\Big(K_G+\Big(1-\frac{d(\epsilon)+2}{r}\Big)\Delta|_G+ \frac{d(\epsilon)+2}{r}D+F_1+F_2\Big) \\
\sim_\bQ {}&-\Big(1-\frac{d(\epsilon)+2}{r}\Big) (K_X+\Delta)|_G
\end{align*}
is ample, where $r>d(\epsilon)+2$ by definition. Hence by Connectedness Lemma, 
$$
\Nklt\Big(G, \Big(1-\frac{d(\epsilon)+2}{r}\Big)\Delta|_G+\frac{d(\epsilon)+2}{r}D+F_1+F_2\Big)
$$
is connected. Since it contains $F_1$ and $F_2$, there exists a non-klt center connecting $F_1$ and $F_2$, and hence dominating $H$. Restricting on a general fiber $F$ of $f|_G:G\to H$, by adjunction, 
$$
\Big(F, \Big(1-\frac{d(\epsilon)+2}{r}\Big)\Delta|_F+\frac{d(\epsilon)+2}{r}D|_F\Big)
$$
is not klt. On the other hand, $(F, (1-\frac{d(\epsilon)+2}{r})\Delta|_F)$ is $\epsilon$-klt since $(X, \Delta)$ is  $\epsilon$-klt. Hence
$$
\deg \Big(\frac{d(\epsilon)+2}{r}D|_F\Big)\geq \epsilon.
$$
Note that 
$$
\deg (D|_F)=\deg( -K_X|_F)=2,
$$
this implies that 
$$
 \frac{2(d(\epsilon)+2)}{r}\geq \epsilon,
$$
which contradicts to the definition of $r$.
\end{proof}
Now we can prove the theorem of this subsection.
\begin{thm}\label{thm conic}
Let $X$ be a $3$-fold of $\epsilon$-Fano type with a Mori fiber structure $f:X\rightarrow S$ to a surface. Then $$\Vol(-K_X)\leq\frac{144(d(\epsilon)+2)}{\epsilon^2}.$$
\end{thm}
\begin{proof}Assume to the contrary that $\Vol(-K_X)>\frac{144(d(\epsilon)+2)}{\epsilon^2}.$ Take a rational number $t$ such that
$$
\Vol(-K_X)>t\cdot \frac{24(d(\epsilon)+2)}{\epsilon}>\frac{144(d(\epsilon)+2)}{\epsilon^2}.
$$
\begin{lem}\label{lemma3}
$-K_X-tG$ is $\bQ$-effective.
\end{lem}
\begin{proof}
For a positive integer $p$ and a sufficiently divisible positive integer $m$, we have an exact sequence
$$
0\rightarrow \OO_X(-mK_X-pG)\rightarrow \OO_X(-mK_X-(p-1)G)\rightarrow \OO_G(-mK_X-(p-1)G)\rightarrow 0.
$$
 Hence
\begin{align*}
{}&h^0(X, \OO_X(-mK_X-pG))\\
\geq{}& h^0(X, \OO_X(-mK_X-(p-1)X))- h^0(G, -mK_X|_G-(p-1)G|_G)\\
\geq{}&  h^0(X, \OO_X(-mK_X-(p-1)X))- h^0(G, -mK_X|_G).
\end{align*}
Here we use the fact that $G|_G\geq 0$.
Inductively, we have
$$
h^0(X, \OO_X(-mK_X-pG))\geq h^0(X, \OO_X(-mK_X))- p\cdot h^0(G, -mK_X|_G).
$$
Now we may take the integer $p=tm$ since $m$ is sufficiently divisible. By the definition of volume, we have
\begin{align*}
{}&\lim_{m\rightarrow \infty}\frac{6}{m^3}\bigg(h^0(X, \OO_X(-mK_X))- tm\cdot h^0(G, \OO_F(-mK_X|_G))\bigg)\\
={}&\Vol(-K_X)-3t\Vol(-K_X|_G)\\
\geq{}&\Vol(-K_X)-\frac{24t(d(\epsilon)+2)}{\epsilon}>0.
\end{align*}
Here we use Lemma \ref{lemma2}.
Hence $h^0(X, \OO_X(-mK_X-tmG))>0$ for $m$ sufficiently divisible, that is, $-K_X-tG$ is $\bQ$-effective. 
\end{proof}
By Lemma \ref{lemma3}, there exists an effective $\bQ$-divisor $B\sim_\bQ-K_X- tG$.
For a general fiber $F$ of $X$ over $z\in S$, there exists a number $\eta>0$ (cf. \cite[4.8]{SOP}) such that for any general $H'\in |H|$ containing $z$, 
$$\Nklt\Big(X, \Big(1-\frac{3}{t}\Big)\Delta+\frac{3}{t}B\Big)=\Nklt\Big(X, \Big(1-\frac{3}{t}\Big)\Delta+\frac{3}{t}B+\eta f^*(H')\Big).$$
We may take general $H_j\in |H|$ containing $z$ for $1\leq j\leq J$ with $J>\frac{2}{\eta}$  and take $G_1=\sum_{j=1}^J\frac{2}{J}f^*(H_j)$. Then $\mult_{F}G_1\geq 2$ and $G_1\sim_\bQ 2f^*(H)\sim_\bQ 2G$. In particular, $(X, G_1)$ is not klt at $F$ and by construction, in a neighborhood of $F$,
\begin{align*}
{}&\Nklt\Big(X, \Big(1-\frac{3}{t}\Big)\Delta+\frac{3}{t}B\Big)\cup F\\
={}&\Nklt\Big(X, \Big(1-\frac{3}{t}\Big)\Delta+\frac{3}{t}B+G_1\Big).
\end{align*}
Take a general element $G_2\in |f^*(H)|$ not containing $F$, consider the pair $(X, (1-\frac{3}{t})\Delta+\frac{3}{t}B+G_1+G_2)$ where $3/t<1$. 
Then
$$-\Big(K_{X}+ \Big(1-\frac{3}{t}\Big)\Delta+\frac{3}{t}B+G_1+G_2\Big)\sim_\bQ-\Big(1-\frac{3}{t}\Big)(K_X+\Delta)$$
is ample.
Since 
$$F\cup G_2\subset \Nklt\Big(X, \Big(1-\frac{3}{t}\Big)\Delta+\frac{3}{t}B+G_1+G_2\Big),$$
 by Connectedness Lemma,   there is a curve $C$ contained in $\Nklt(X, (1-\frac{3}{t})\Delta+\frac{3}{t}B+G_1+G_2)$, intersecting $F$ and not contracted by $f$. Hence $C$ is  contained in $\Nklt(X, (1-\frac{3}{t})\Delta+\frac{3}{t}B)$ by the construction of $G_1$ and generality of $G_2$. Since $C$ intersects $F$, so does $\Nklt(X, (1-\frac{3}{t})\Delta+\frac{3}{t}B)$. 
 By inversion of adjunction, 
$(F, (1-\frac{3}{t})\Delta|_F+\frac{3}{t}B|_F)$ is not klt for a general fiber $F$. On the other hand, $(F, \Delta|_F)$ is $\epsilon$-klt and $F\simeq \bP^1$. Hence  
$$
\deg \Big(\frac{3}{t}B|_F\Big)\geq \epsilon.
$$
Note that $$
\deg (B|_F)=\deg( -K_X|_F)=2,
$$
Hence $t\leq \frac{6}{\epsilon}$, which contradicts the definition of $t$.
Hence we proved Theorem \ref{thm conic}
\end{proof}

\end{document}